\documentclass{amsart}
\usepackage {amssymb}
\usepackage {amsmath}
\usepackage{amsthm}
\usepackage{graphicx}
\usepackage {amscd}
\usepackage {epic}
\usepackage[alphabetic]{amsrefs}
\usepackage[colorlinks, citecolor=blue]{hyperref}
\newcommand{\red}[0]{\operatorname{red}}

\newcommand{\dlt}[0]{\operatorname{dlt}}

\newcommand{\DMR}{{\mathcal{DMR}}}
\newcommand{\DR}{{\mathcal{DR}}}
\newcommand{\D}{{\mathcal{D}}}

\newcommand{\mult}{{\rm mult}}

\newcommand{\ord}{{\rm ord}}

\newcommand{\Link}{{\rm Link}}
\newcommand{\lct}{{\rm lct}}
\newcommand{\vol}{{\rm vol}}
\newcommand{\Val}{{\rm Val}}
\newcommand{\fa}{{\mathfrak{a}}}
\newcommand{\hvol}{\widehat{\rm vol}}
\newcommand{\Cal}{\mathcal}
\newcommand{\bb}{\mathbb}

\newtheorem{thm}{Theorem}[section]

\newtheorem{lem}[thm]{Lemma}
\newtheorem{cor}[thm]{Corollary}

\newtheorem{conj}[thm]{Conjecture}

\theoremstyle{definition}
\newtheorem{defn}[thm]{Definition}

\newtheorem{exmp}[thm]{Example}

\newtheorem{ques}[thm]{Question}    

\newtheorem{rem}[thm]{Remark}

\newtheorem{defn-thm}[thm]{Definition--Theorem}  
\newtheorem{defn-prop}[thm]{Definition--Proposition}  
\newtheorem{defn-lem}[thm]{Definition--Lemma}  

\theoremstyle{remark}

\usepackage{fancyhdr}
\begin{document}
\title{Interaction between singularity theory and the minimal model program }

\author {Chenyang Xu}

\address   {Beijing International Center for Mathematical Research,
       Beijing 100871, China}
\email     {cyxu@math.pku.edu.cn}
\begin{abstract} 
We survey some recent topics on singularities, with a focus on their connection to the minimal model program. This includes the construction and properties of dual complexes, the proof of the ACC conjecture for log canonical thresholds and  the recent progress on the `local stability theory'  of an arbitrary Kawamata log terminal singularity. 
\end{abstract}

\date{\today}

\maketitle{}
\section{Introduction}
Through out this paper,  we will consider algebraic singularities in characteristic 0.
It is well known that even if we are mostly interested in smooth varieties, for many different reasons, we have to deal with singular varieties. For the minimal model program (MMP) (also known as Mori's program), the reason is straightforward, mildly singular varieties are built into the MMP  process, and there is no good way to avoid them (see e.g. \cite{KM98}). In fact, the development of the MMP has been intertwined with the progress of our understanding of the corresponding singularity theory,  in particular for the classes of singularities preserved by an MMP sequence. This is one of the main reasons why when the theory was started around four decades ago, people spent a lot of time to classify these singularities. However, once we move beyond dimension three, an explicit characterisation of these singularities is often too complicated, and we have to search for a more intrinsic and qualitative method. It turns out that MMP theory itself provides many new tools for the study of singularities. In this note, we will survey some recent progress along these lines. More precisely, we will discuss the construction and properties of dual complexes, the proof of the ACC conjecture for log canonical thresholds, and the recently developed concept of `local stability theory' of an arbitrary Kawamata log terminal singularity. We hope these different aspects will give the reader an insight to the modern philosophy of studying singularities from the MMP viewpoint.

\medskip

In the rest of the introduction, we will give a very short account on some of the main ideas. Given a singularity $x\in X$ in characteristic 0, the first birational model that one probably thinks of is a smooth one given by Hironaka's theorem on resolution of singularities. However, started from dimension three, there are often too many possible resolutions and examples clearly suggest that in a general case, an `optimal resolution' does not exist.  By the philosophy of MMP, we should run a sequence of relative MMP, which allows us to start from a general birational model over $x\in X$, e.g., an arbitrary resolution, and produce a sequence of relative birational models. The output of this MMP is a birational model, which usually is mildly singular but still equipped with many desirable properties. Furthermore, since during the MMP process,  each step is  a simple surgery like a divisorial contraction or a flip, we can keep track of many properties of the models and use this information to answer questions. As an example, in Section \ref{s-dualcomplex}, we will consider the construction of a CW-complex as a topological invariant for an isolated singularity $x\in X$ with $K_X$ being $\mathbb{Q}$-Cartier, namely the dual complex of a minimal resolution denoted by $\DMR(x\in X)$.

A possibly more profound principle is that there is a local-to-global analogue between different types of singularities and the building blocks of varieties. More precisely, the MMP can be considered as a process to transform and decompose an arbitrary projective variety into three types,  which respectively have  positive (Fano), zero (Calabi-Yau) or negative (KSBA) first Chern class. These three classes are naturally viewed as building blocks for higher dimensional varieties. As a local counterpart, we consider normal singularities whose canonical class is $\bb Q$-Cartier. There is a closely related  trichotomy: the minimal log discrepancy is larger, equal or smaller than 0. In fact, guided by the local to global principle, we are able to discover  striking new results on singularities. In Section \ref{s-acc}, we will focus on the proof of Shokurov's ACC conjecture on log canonical thresholds, which is achieved via an intensive interplay between local and global geometry.  In Section \ref{s-klt}, we will investigate in a new perspective on Kawamata log terminal (klt) singularities which are precisely the singularities with positive log discrepancies and form the local analog of Fano varieties.  We will explain some deep insights on klt singularities inspired by advances in the study of Fano varieties. More precisely, for Fano varieties, we have the notion of K-(semi,poly)stability  which has a differential geometry origin, as it is expected to characterise the existence of a K\"ahler-Einstein metric. 
For  klt singularities, the local to global principle leads us to discover a (conjectural) stability theory, packaged in  the {\it Stable Degeneration Conjecture}  \ref{conj-local}, which can be considered as a local analogue to the K-stability for Fano varieties.

\medskip

 \noindent {\it Reference:} Giving a comprehensive account of the relation between the singularity theory and the MMP is far beyond the scope of this note.  The singularity theory in the MMP is extensively discussed in the book \cite{Kol13}. Ever since Koll\'ar's book was published,  many different aspects of singularity theory have significantly evolved, and important new results have been established.
   
 \bigskip  
 \noindent {\bf Acknowledgement:}   I want to thank my coauthors, especially Christopher Hacon,  J\'anos Koll\'ar, Chi Li and James M$^{\rm c}$Kernan for many discussions and joint works on the mathematical materials surveyed in this note. I am grateful to Christopher Hacon and J\'anos Koll\'ar for long lists of valuable comments. CX is partially supported by `The National Science Fund for Distinguished Young Scholars (11425101)'.
   
 \section{Dual complex}\label{s-dualcomplex}
 There are many standard references in the subject of MMP, see e.g. \cite{KM98}. Here we recall some basic definitions. 
 Given a normal variety $X$ and a $\mathbb{Q}$-divisor $\Delta$ whose coefficients along prime components are contained in $\mathbb{Q}\cap [0,1]$, we call $(X,\Delta)$ a {\it $\bb Q$-Cartier log pair} if $K_X+\Delta$ is $\mathbb{Q}$-Cartier, e.g. there is some positive integer $N$ such that $N(K_X+\Delta)$ is Cartier. For a divisorial valuation $E$ whose centre on $X$ is non-empty, we can assume there is a birational model $f\colon Y\to X$, such that $E$ is a divisor on $Y$. Then we can define the {\it discrepancy} $a(E,X,\Delta)$ for a $\mathbb{Q}$-Cartier log pair $(X,\Delta)$ to be the multiplicity of 
 $$K_{Y/X}+f^*\Delta=K_Y-f^*(K_X+\Delta)$$ along $E$. This is a rational number of the form $\frac{p}{N}$ for some integer $p$. For many questions, it is more natural to look at the {\it log discrepancy} 
 $A(E,X,\Delta)=a(E,X,\Delta)+1$, which is also denoted by $A_{X,\Delta}(E)$ in the literature.  We say that $(X,\Delta)$ is {\it log canonical} (resp. {\it Kawamata log terminal} ({\it klt})) if $A(E,X,\Delta)\ge 0$ (resp. $A(E,X,\Delta)> 0$) for all divisorial valuations $E$ whose centre ${\rm Center}_X(E)$ on $X$ is non-empty.  There is another important class called {\it divisorial log terminal} ({\it dlt}) sitting in between: a log pair $(X,\Delta)$ is dlt if there is a smooth open locus $U\subset X$, such that $\Delta_U=_{\rm defn} \Delta|_U$ is a reduced divisor satisfying $(U,\Delta_U)$ is simple normal crossing, and any divisor $E$ with the centre ${\rm Center}_X(E)\subset X\setminus U$ satisfies $A(E,X,\Delta)>0$. 
 The main property for the discrepancy function is that $a(E,X,\Delta)$ monotonically increases under a MMP sequence, which implies that the MMP will preserve the classes of singularities defined above (cf. \cite[3.42-3.44]{KM98}).

We call a projective variety $X$ to be {\it a $\mathbb{Q}$-Fano variety} if $X$ only has klt singularities and $-K_X$ is ample. Similarly, a projective pair $(X,\Delta)$ is called {\it a log Fano pair} if $(X,\Delta)$ is klt and $-K_X-\Delta$ is ample.
 
 \subsection{Dual complex as PL-homeomorphism invariant} 
 For a simple normal crossing variety $E$, it is natural to consider how the components intersect with each other. This combinatorial data is captured by the dual complex $\mathcal{D}(E)$ (see Definition \ref{d-dual}). A typical example one can keep in mind is the dual graph $\mathcal{D}(E)$ for a resolution $(Y,E)\to X$ of a normal surface singularity, where the exceptional curve $E$ is assumed to be of simple normal crossings. This invariant  is indispensable for the study of surface singularities (see e.g. \cite{Mum61}). 
 Nevertheless, the concept of dual complex can be defined in a more general context. 
  \begin{defn}[Dual Complex]\label{d-dual}
  Let $E=\bigcup_{i\in I}E_i$ be a pure dimensional scheme with
irreducible components $E_i$. Assume that
\begin{enumerate}
\item each $E_i$ is normal  and
\item for every $J\subset I$, if $\cap_{i\in J}E_i$ is nonempty, then every connected component
of  $\cap_{i\in J}E_i$ is  irreducible and has codimension $|J|-1$ in $E$. 
\end{enumerate}
Note that  assumption (2) implies the following.
\begin{enumerate}\setcounter{enumi}{2}
\item For every $j\in J$, 
every irreducible component
of $\cap_{i\in J}E_i$ is contained in a 
unique irreducible component
of $\cap_{i\in J\setminus\{j\}}E_i$.
\end{enumerate}
The {\it dual complex} $\D(E)$ of $E$ is the regular cell complex obtained
as follows.
The vertices are the irreducible components of $E$ and
to each irreducible component
of $W\subset \cap_{i\in J}E_i$ we associate a  cell of dimension $|J|-1$.
This cell is usually denoted by $v_W$. The attaching map is given by condition (3).
Note that $\D(E)$ is a  simplicial complex
iff $\cap_{i\in J}E_i$ is irreducible  (or empty)
for every $J\subset I$. 
 \end{defn}

 Fixed a dlt pair $(X,\Delta)$, the reduced part $E=_{\rm defn}\Delta^{=1}$ of $\Delta$ satisfies the assumptions in Definition \ref{d-dual} (see e.g. \cite[Section 4.2]{Kol13}), thus we can define $\D(X,\Delta)=_{\rm defn}\D(E)$. Clearly, by the definition of $U$ in the definition of dlt singularity $(X,\Delta)$,  we can pick any such $U$,  then $\D(X,\Delta)=\mathcal{D}(\Delta|_U)$. Furthermore, if two dlt pairs $(X,\Delta)$ and $(X',\Delta')$ are crepant birationally equivalent, i.e., the pull backs of $K_X+\Delta$ and $K_{X'}+\Delta'$ to a common model are the same, then applying the weak factorisation theorem to log resolutions of $(X,\Delta)$ and $(X',\Delta')$ and carefully tracking the dual complex given by divisors with log discrepancy 0 on each birational model, we can show that $\mathcal{D}(X,\Delta)$ and $\mathcal{D}(X',\Delta')$ are PL-homeomorphic (see \cite[11]{dFKX}). 

 Given a sequence of MMP
 $$(X_1,\Delta_1)\dasharrow (X_2,\Delta_2)\dasharrow \cdots \dasharrow (X_k,\Delta_k) ,$$
 as the log discrepancies of $(X_i,\Delta_i)$ monotonically increase, we have 
 $$\mathcal{D}(X_1,\Delta_1)\supset \mathcal{D}(X_2,\Delta_2)\supset\cdots \supset \mathcal{D}(X_k,\Delta_k).$$ 
 
 \begin{rem}Although part of the MMP, including the abundance conjecture, remains to be conjectural, all the MMP results we need in this note are already proved in \cite{BCHM10} and its extensions, e.g. \cite{HX13}.
 \end{rem}
 
 We know the following technical but useful criterion. 
 \begin{lem}[{\cite[19]{dFKX}}]\label{l-htp}
 If  for  a step  $(X_i,\Delta_i)\dasharrow (X_{i+1},\Delta_{i+1})$ of an MMP sequence, the extremal ray  $R_i$ satisfies that $R_i\cdot D_i>0$ for  a component $D_i$ of $\Delta^{=1}_{i}$, then $\mathcal{D}(X_i,\Delta_i)\supset \mathcal{D}(X_{i+1},\Delta_{i+1})$ is a homotopy equivalence. 
\end{lem}
 
 Now we can apply this to various geometric situations.  We first consider the application to the study of a singularity $x\in X\subset \mathbb{C}^N$. It has been known for long time (see \cite{Mil68}) that all local topological information of $x\in X$ is encoded in the link defined as 
 $$\Link(x\in X)=_{\rm defn}X \cap B_{\epsilon}(x)$$ for a sufficiently small radius $\epsilon$. Following the strategy of studying surfaces (as in e.g. \cite{Mum61}), we pick a log resolution $Y \to (x\in X)$ and let $E=_{\rm defn}f^{-1}(x))$ (in particular, $E$ is simple normal crossing). Then $\Link(x\in X)$ is a tubular neighbourhood of $E$ and $\D(E)$ contains some key information of this tubular structure. 
  
 \begin{exmp}\label{e-ADE}
 Consider the well known classification of rational double points (or Du Val singularities) on surface:
 \begin{enumerate}
 \item Type $A_n$: $x^2+y^2+z^{n+1}=0.$
 \item Type $D_n$: $x^2+zy^2+z^{n-1}=0$ $(n\ge 4) $.
\item Type $E_6$:  $ x^{2}+y^{3}+z^{4}=0$.
\item Type $E_7$: $x^{2}+y(y^{2}+z^{3})=0$
\item Type $E_8$:  $x^{2}+y^{3}+z^{5}=0$.
 \end{enumerate}
 Then the minimal resolution $Y$ with the exceptional locus $E$ forms a log resolution, and $\D(E)$ is the graph underlying the corresponding Dynkin diagram. 
 \end{exmp}
 
 Using the weak factorisation theorem \cite{AKMW02}, one shows that the homotopy class of $\mathcal{D}(E)$ is a well-defined homotopy invariant $\DR(x\in X)$ which does not depend on the choice of the log resolution $(Y,E)$ (see e.g. \cite{Pay13}). 
 The strategy in our previous discussion then can be used to show the following result.  
 \begin{thm}\label{t-mmpdual} For  an isolated normal singularity $x\in X$ with $K_X$ being $\bb Q$-Cartier, we can define a canonical PL-homeomorphism invariant $\DMR(x\in X)$ which has the homotopy class of $\DR(x\in X)$.
 \end{thm}
 \begin{proof} First we take a log resolution $(Y,E)\to X$ which is isomorphic outside $X\setminus \{x\}$, then run a relative MMP of $(Y,E)$ over $X$. The output $(X^{\dlt},\Delta^{\dlt})$ is called {\it a dlt modification} of $x\in X$. Then we define a regular complex 
 $$\DMR(x\in X)=_{\rm defn}\D(\Delta^{\dlt}).$$ Since different dlt modifications are crepant birationally equivalent to each other, we know $\DMR(x\in X)$ gives a well defined PL-homeomorphism class by the discussion before. Furthermore, we can deduce from Lemma \ref{l-htp} that for all the birational models appearing in steps of the relative MMP, including the last one $(X^{\dlt},\Delta^{\dlt})$, the dual complexes have the same homotopy type. Then it implies $\DMR(x\in X)$ is homotopy equivalent to $\DR(x\in X)$. 
 \end{proof}
The regular complex $\DMR(x\in X)$ can be considered as a geometric realisation of the weight 0 part of the Hodge theoretic invariant attached to $x\in X$. An interesting corollary to Theorem \ref{t-mmpdual} is that if we consider a klt singularity $x\in X$, then $\DR(x\in X)$ is contractible, as in the special case of Example \ref{e-ADE}. 
\medskip

Another natural context in which the dual complex appears is for the motivic zeta function using the  log resolution formula (cf. \cite[Section 3]{DL01}). The techniques developed here can be used to show that the only possible maximal order  pole of the motivic zeta function  is the negative of the log canonical threshold, which was a conjecture by Veys (see \cite{NX16a}).

 \subsection{Dual complex of log Calabi-Yau pairs}
 
 Similar ideas can be applied when we consider the setting of a proper degeneration $Y\to C$ of projective varieties over a smooth pointed curve $(C,0)$. Here we consider the dual complex $\mathcal{D}(Y^{\rm red}_0)$ where $Y_0^{\rm red}$ is the reduced fiber over $0$ and assume  $(Y, Y_0^{\rm red})$ is a dlt pair.  
 
 For subjects like mirror symmetry,  the degeneration of Calabi-Yau varieties is of particular interests. From a birational geometry view, if we consider a family $\pi\colon Y\to C$, and assume a general fiber has $K_{Y_t}\sim_{\mathbb{Q}} 0$, then after running an MMP over $C$ (cf. \cite{Fuj11}), we end up with a model $Y$ which satisfies $K_Y+Y_0^{\rm red}\sim_{\mathbb{Q}}0$. Then for such models with this extra condition,  any two of them are crepant birationally equivalent which implies $\D({Y}_0^{\red})$ is well defined up to PL-homeomorphism. 
 
Indeed in this case, the topological invariant  $\D({Y}_0^{\red})$ is first defined by Kontsevich-Soibelman as the `essential skeleton' of the Berkovich theoretic non-archimedean analytification $Y^{\rm an}$ (see \cite{KS01,MN15, NX16}), and it plays an important role in the study of the algebro-geometric version of the SYZ conjecture (cf. \cite{SYZ96, KS01, GS11} etc.). The same argument for proving Theorem \ref{t-mmpdual}  can be used to show the essential skeleton $\D({Y}_0^{\red})$ is homotopy equivalent to $Y^{\rm an}$ (see \cite{NX16}).
 
 \bigskip
 
To understand $\D({Y}_0^{\red})$, we first need to describe its local structure, i.e., for a prime component $X\subset Y_0$, to understand the link of the corresponding vertex $v_{X}$ in $\D({Y}_0^{\red})$. It is given by $\D(\Delta^{=1})$ where $\Delta$ is defined by formula 
$$(K_Y+Y^{\red}_0)|_X=K_X+\Delta\sim_{\mathbb{Q}}0.$$ Our goal is to show under suitable conditions, the dual complex coming from a (log) Calabi-Yau variety is close to simple objects like a sphere or a disc.  In fact, we can show the following.
 
 \begin{thm}[\cite{KX16}]\label{t-logCY}
 Let $(X,\Delta)$ be a projective dlt pair which satisfies $K_X+\Delta\sim_{\mathbb{Q}}0$. Assume $\dim(\D(\Delta^{=1}))>1$, then the following holds.
 \begin{enumerate}
  \item  $H^i(\D(\Delta^{=1}),\mathbb{Q})=0$ for $1\le i \le \dim(\D(\Delta^{=1},\mathbb{Q}))$.
 \item $\D(\Delta^{=1})$ is a pseudo-manifold with boundary (\cite{KK10}).
 \item There is a natural surjection $\pi_1(X^{\rm sm})\twoheadrightarrow \pi_1(\D(\Delta^{=1}) )$.
 \item The profinite completion $\hat{\pi}_1(\D(\Delta^{=1}))$ is finite. 
 \end{enumerate}
 \end{thm}
 \begin{proof} We sketch the proof under the extra assumption that $\dim(\D(\Delta^{=1}))$ is maximal, i.e. equal to $n-1$.
Then (1) is easily obtained using Hodge theory. We need to apply MMP theory to show (2) and (3). Here we explain the argument for (3). A carefully chosen MMP process (see \cite[Section 6]{KX16}) allows us to change the model from $X$ to a birational model $X'$, with the property that if we define the effective $\mathbb{Q}$-divisor $\Delta'$ on $X'$ to be the one such that $(X,\Delta)$ and $(X',\Delta')$ are crepant birationally equivalent, then the support of $\Delta'$ contains an ample divisor. From this we conclude that
$$\pi_1(X'^{\rm sm})\twoheadrightarrow  \pi_1(\Delta'^{=1})\twoheadrightarrow  \pi_1(\D(\Delta^{=1})),$$
where the first surjection follows from the (singular version of) Lefschetz Hyperplane Theorem.
We also have $\pi_1(X^{\rm sm})\to \pi_1(X'^{\rm sm})$ by tracking the MMP process, which concludes (3).   Finally, (4) follows from \cite{Xu14}.
(We note here that we indeed expect $\pi_1(\D(\Delta^{=1}))$ is finite, but to apply the above argument, we need $\pi_1(X^{\rm sm})$ for the underlying variety $X$ of a log Fano pair. For now, we only know its pro-finite completion is finite.) 
\end{proof}
A remaining challenging  question is to understand the torsion cohomological  group $H^i(\D(\Delta^{=1}),\mathbb{Z})$ for a dlt log Calabi-Yau pair $(X,\Delta)$. 
  
 \section{ACC of log canonical thresholds}\label{s-acc}
 
 Given a holomorphic function $f$ with $f(0)=0$, the complex singular index
 $$c(f)=\sup\{c\  |\  \frac{1}{|f|^c} \mbox{ is locally $L^2$-integrable at $0$}\}$$
introduced by Arnold is a fundamental invariant, which appears in many contexts (see \cite[II. Chap. 13]{AGV85}). In a more general setting, in birational geometry, this invariant is interpreted as the {\it log canonical threshold} of an effective $\mathbb{Q}$-divisor $D$ with respect to a log pair $(X,\Delta)$ 
 $$\lct(X,\Delta; D)=\max\{t\ |\  (X,\Delta+tD) \mbox{ is log canonical}\}.$$
 Using a log resolution, it is not hard to show that when $X$ is the local germ $0\in \mathbb{C}^n$, $\Delta=0$ and $D=(f)$, $\lct(X;D)=c(f)$. 
 
 \begin{exmp}Let $X=\mathbb{C}^n$, $f=x_1^{m_1}+\cdots+ x_n^{m_n}$, then by \cite[8.15]{Kol95}
 $$\lct(\mathbb{C}^n,f)=\min\{1, \sum^n_{i=1}\frac{1}{m_i}\}.$$
 For a fixed $n$, all such numbers form an infinite set which satisfies the ascending chain condition.
 \end{exmp}
 
See \cite[Section 8-10]{Kol95} for a wonderful survey, including relations with other branches of mathematics.  
 
 We define the following set.
 \begin{defn} Fix the dimension $n$ and two sets of positive numbers $I$ and $J$, we denote by ${\rm LCT}_n(I,J)$  the set consisting of all numbers $\lct(X,\Delta;D)$ such that $\dim(X)=n$, the coefficients of $\Delta$ are in $I$ and the coefficients of $D$ are in $J$.
\end{defn} 
Our main contribution to the study of log canonical thresholds is showing the following  theorem.
\begin{thm}[{\cite[Theorem 1.1]{HMX14}}, ACC Conjecture for log canonical thresholds]\label{t-ACC}
 If $I$ and $J$ satisfy the descending chain condition (DCC), then ${\rm LCT}_n(I,J)$ satisfies the ascending chain condition (ACC). 
 \end{thm}
 
 In such a generality, this was conjectured in \cite{Sho}, although in a lot of earlier works, questions of a similar flavour already appeared. For $X=\mathbb{C}^n$ (or even more generally for bounded singularities) and  $D=(f)$, this was solved by \cite{dFEM} using a different approach. In fact, while our proof is via global geometry, the argument in \cite{dFEM} uses a more local method. 

\medskip
To understand our strategy, we start with a well-known construction:
 Given a log canonical pair $(X,\Delta)$ with a prime divisor $E$ over $X$ with the log discrepancy $A(E,X,\Delta)=0$, if $X$ admits a boundary $\Delta'$ such that $(X,\Delta')$ is klt, then applying the MMP we can construct a model $f\colon Y\to X$ such that ${\rm Ex}(f)$ is equal to the divisor $E$ (see \cite[1.4.3]{BCHM10}). Denote by $\Delta_Y=E+f_*^{-1}\Delta$ and restrict $K_Y+\Delta_Y$ to a general fiber $F$ of $f\colon E\to f(E)$. Since $E$ has coefficient one in $\Delta_Y$, the adjunction formula says there is a boundary $\Delta_F$ such that 
 $$K_F+\Delta_F=(K_{Y}+\Delta_Y)|_F=f^*(K_X+\Delta)\sim_{\mathbb{Q}}0.$$
 In other words, using the model $Y$ constructed by an MMP technique, from a lc pair $(X,\Delta)$ which is not klt along a subvariety $f(E)$, we obtain a log Calabi-Yau pair $(F,\Delta_F)$ of smaller dimensional. 
 
  We note that even in the case $\Delta_Y=E$, since $Y$ could be singular along codimension 2 points on $E$, it is not always the case that $\Delta_E=0$. Nevertheless, if the coefficients of $\Delta$ are in a set $I\subset [0,1]$, then the coefficients of $\Delta_F$ are always in the set 
 $$D(I)=_{\rm defn}\{\frac{n-1+a}{n} \ |\ n\in \mathbb{N}, a=\sum^j_{i=1} a_i \mbox{ where }a_i\in I\}\cap [0,1] $$
 (see e.g. \cite[3.45]{Kol13}).
 In particular, if $I$ satisfies the DCC, then $D(I)$ satisfies the DCC. This is why we work with such a general setting of coefficients as it works better with the induction. 
 
 Moreover, if there is a sequence of pairs  $(X_i,\Delta_i)$ and strictly increasing log canonical thresholds $t_i$ with respect to the divisors $D_i$, then the above construction will produce a sequence of log Calabi-Yau varieties $(F_i, \Delta_{F_i})$ corresponding to $(X_i,\Delta_i+t_iD_i)$ with the property that the restriction of $f_*^{-1}(\Delta_i+t_iD_i)$ on $F_i$ yields components of $\Delta_{F_i}$ with strictly increasing coefficients as $i\to \infty$. Therefore, to get a contradiction,  it suffices to prove the following global version of the ACC conjecture.
 
 \begin{thm}[{\cite[Theorem 1.5]{HMX14}}]\label{t-gacc}
 Fix $n$ and a DCC set $I$, then there exists a finite set $I_0\subset I$ such that for any projective $n$-dimensional log canonical Calabi-Yau pair $(X,\Delta)$, i.e. $K_X+\Delta\sim_{\mathbb{Q}}0$, with the coefficients of $\Delta$ contained in $I$, it indeeds holds that the coefficients of $\Delta$ are in $I_0$.
 \end{thm}
 We note that Theorem \ref{t-gacc} in dimension $n-1$ implies Theorem \ref{t-ACC} in dimension $n$. More crucially, 
 Theorem \ref{t-gacc} changes the problem from a local setting to a global one and we have many new tools to study it. In particular, as we will explain below, Theorem \ref{t-gacc} relates to the boundedness results on log general type pairs. This is a central topic in the study of such pairs, especially for the construction of the compact moduli space of KSBA stable pairs, which is the higher dimensional analogue of the moduli space of marked stable curves $\overline{\mathcal{M}}_{g,n}$ (see e.g. \cite{Kol13a,HMX16}).
  
 Since $I$ satisfies the DCC, if such a finite set $I_0$ does not exist, we can construct an infinite sequences  $(X_i,\Delta_i)$ of log canonical Calabi-Yau pairs of dimension at most $n$, such that after reordering, if we write $\Delta_i=\sum^k_{j=1} a_i^j\Delta^j_i$, $\{a_i^j\}^{\infty}_{i=1}$ monotonically increases for any fixed $1\le j\le k$ and  strictly increases for  at least one. Furthermore, after running an MMP, we can reduce to the case that the underlying  variety $X_i$ is a Fano variety with the Picard number $\rho(X_i)=1$. Then if we push up the coefficients of $\Delta_i$ to get a new boundary 
 $$\Delta'_i=_{\rm defn}\sum^k_{j=1}{a^{j}_{\infty}}\Delta^j_i\qquad \mbox{ where } a^j_{\infty}=\lim_i a^j_i,$$
 $K_{X_i}+\Delta_i'$ is ample. By enlarging $I$, we can start with the assumption that all accumulation points of $I$ are also contained $I$. In particular, the coefficients $a^j_{\infty}\in I$. Moreover,  recall that by induction on the dimension, we can assume Theorem \ref{t-gacc} holds for dimension $n-1$, which implies Theorem \ref{t-ACC} in dimension $n$. Thus for $i$ sufficiently large, $(X_i,\Delta_i')$ is also log canonical. 
 Then we immediately get a contradiction to the second part of (2) in the following theorem.
 \begin{thm}[{\cite[Theorem 1.3]{HMX14}}]\label{t-bound}
 Fix dimension $n$ and a DCC set $I\subset [0,1]$. Let $\mathfrak{D}_n(I)$ be the set of all pairs 
$$ \{(X,\Delta)\ |\ \dim(X)=n,   (X,\Delta) \mbox{ is lc and the coefficients of $\Delta$ are in $I$}\}, $$
and $\mathfrak{D}^{\circ}_n(I)\subset \mathfrak{D}_n(I)$ the subset of pairs with $K_X+\Delta$ being big. Then the following holds.
 \begin{enumerate}
 \item The set ${\rm Vol}_n(I)=\{{\rm vol}(K_X+\Delta)|\  (X,\Delta)\in \mathfrak{D}_n(I) \}$ satisfies DCC.
 \item There exists a positive integer $N=N(n,I)$ depending  on $n$ and $I$ such that the linear system $|N(K_X+\Delta)|$ induces a birational map for any $(X,\Delta)\in \mathfrak{D}^{\circ}_n(I)$. 
 Moreover, there exists $\delta>0$ depending only on $n$ and $I$, such that if $(X,\Delta)\in \mathfrak{D}^{\circ}_n(I)$, then $K_X+(1-\delta)\Delta$ is big. 
 \end{enumerate}
 \end{thm}
The part (1) was a conjecture of Alexeev-Koll\'ar (cf. \cite{Kol94, Ale94}). As already mentioned, it is the key in the proof of the boundedness of  the moduli space of KSBA stable pairs with fixed numerical invariants. See \cite{HMX16} for a survey on this topic and related literature. 
 
 \medskip
 
 During the proof of Theorem \ref{t-bound}, we have to treat (1) and (2) simultaneously.  Such a strategy was first initiated in \cite{Tsuji}, and carried out by \cite{HM06, Tak06} for $X$ with canonical singularities and $\Delta=0$. It started with the simple observation that for smooth varieties of general type, birationally boundedness implies boundedness (after the MMP is settled). In \cite{HMX13}, we prove a log version of this, which says that log birational boundedness essentially implies Theorem \ref{t-bound}. This is significantly harder, and we use ideas from \cite{Ale94} which established the two dimensional case of Theorem \ref{t-bound}.  After this, it remains to show that all pairs in $\mathfrak{D}^{\circ}_n(I)$ with the volume bounded from above by an arbitrarily fixed constant is always log birationally bounded, which is done in \cite{HMX14}. One key ingredient is to produce appropriate boundaries on the log canonical centres such that the classical techniques of inductively cutting log canonical centres initiated in \cite{AS95} can be followed here.  
  
  \medskip
  
 Addressing the proof for the ACC of the log canonical thresholds in this circle of global questions is a crucial idea in our solution to it. In fact, in the pioneering work \cite{MP}, an attempt was already made to establish a connection between the ACC of log canonical thresholds and a global question on boundedness but for the set of $K_X$-negative varieties, i.e. Fano varieties. More precisely, it has been shown in \cite{MP} that the ACC conjecture of log canonical thresholds is implied by Borisov-Alexeev-Borisov (BAB) conjecture which is about the bounededness of Fano varieties with a uniform positive lower bound on log discrepancies. More recently, the BAB conjecture is proved in \cite{Bir16}. 
 
  \medskip
 
 In \cite{HMX14}, under a suitable condition on $I$ and assuming that $J=\{\bb N\}$, we show that the accumulation points of ${\rm LCT}_n(I)=_{\rm defn}{\rm LCT}_n(I, \bb N)$ are contained in ${\rm LCT}_{n-1}(I)$, confirming the Accumulation Conjecture due to Koll\'ar.  

  \medskip
 
It attracts considerable interests to find out the effective bound for the constants appearing in Theorem \ref{t-ACC} and \ref{t-bound}. So far it is only successful for low dimension. For instance when $I=0$, Theorem \ref{t-ACC} implies that there exists an optimal $\delta_n<1$ such that ${\rm LCT}_n=_{\rm defn}{\rm LCT}(\{0\})\subset [0,\delta_n]\cup \{1\}$, and $(\delta_n,1)$ is called {\it the $n$-dimensional gap}. It is known $\delta_2=\frac 5 6$, but $\delta_3$ is unknown. In \cite[8.16]{Kol95}, it is asked whether
$$\delta_n=1-\frac{1}{a_n}\qquad \mbox{where $a_1=2$, $a_i=a_1\cdots a_{i-1}$+1}.$$   Our approach in general only gives the existence of $\delta_n$.

 \begin{rem}[ACC Conjecture on minimal log discrepancy]
 There is another deep conjecture about ACC properties of singularities due to Shokurov, which seems to be still far open. 
 
 \medskip
 
\noindent {\it ACC  Conjecture of mld}: Given a log canonical singularity $x\in (X,\Delta)$, we can define 
 $${\rm mld}_{X,\Delta}(x)= \min \{A_{X,\Delta}(E)\ | \ {\rm Center}_E(X)=\{x\}\  \}.$$
 If we fix a finite set $I$, and it is conjectured that the set
 $${\rm MLD}_n(I)=\{ {\rm mld}_{X,\Delta}(x)| \ \dim(X)=n, \mbox{ coefficients of $\Delta$ are in I} \}$$
 satisfies the ACC. 
 
  \medskip
 
 However, compared to  the log canonical thresholds, what kind of global questions connect to this conjecture still remains to be in a myth. For instance, it is not clear which special geometric structure is carried by a divisor attaining the minimal log discrepancy.
 \end{rem}
 
 \section{Klt singularities and K-stability}\label{s-klt}
 
In this section, our discussion will focus on klt singularities. When klt singularities were first introduced, they appeared to be just a technical tool to prove results in the MMP. However, it has become more and more clear that the klt singularities form a very interesting class of singularities, which naturally appears in many context besides the MMP such as constructing K\"ahler-Einstein metrics of Fano varieties etc..

In particular, philosophically,  it has been clear that there is an analogy between klt singularities and Fano varieties. Traditionally, people often prove some properties for an arbitrary Fano variety, then figure out what they imply for the cone singularity over a Fano variety, and finally generalise the statements  to any klt singularity. Only after the corresponding MMP results are established (e.g. \cite{BCHM10}), such analogy can be carried out  in a more concrete manner by really attaching suitable global objects, e.g. Fano varieties, to the singularities.  The first construction was the plt blow up (cf. e.g. \cite{Xu14}) which for a given  klt singularity $x\in (X,\Delta)$, one constructs a birational model $f\colon Y\to X$ such that $f$ is isomorphic outside $x$, $f^{-1}(x)$ is an irreducible divisor $S$, and $(Y,S+f_*^{-1}\Delta)$ is plt. We can also assume $-S$ is ample over $X$, and then $(S,\Delta_S)$ is a log Fano pair, where 
 $$K_S+\Delta_S=_{\rm defn} (K_Y+S+f_*^{-1}\Delta)|_S.$$ 

The divisor $S$ in this construction is called {\it a Koll\'ar component}. It was used to show some local topological properties of $x\in X$ including $\DR(x\in X)$ is contractible (\cite{dFKX}), and the pro-finite completion $\hat{\pi}^{\rm loc}_1(x\in X)$ of the local fundamental group 
$${\pi}^{\rm loc}_1(x\in X)=_{\rm defn}\pi_1({\rm Link}(x\in X))$$
 is finite (\cite{Xu14, TX17}). However, given a klt singularity, usually there could be many Koll\'ar components over it. Only until the circle of ideas of local stability were introduced in \cite{Li15}, a more canonical  picture, though some parts still remain conjectural, becomes clear.  In what follows  we give a survey on this topic. 

\begin{defn}[Valuations] Let $R$ be an $n$-dimensional regular local domain essentially of finite type over a ground field $k$ of characteristic zero. Then a  (real) valuation $v$ of $K={\rm Frac}(R)$ is any map $v \colon K^* \to \mathbb{R}$
which satisfies the following properties for all $a, b$ in $K^*$:
\begin{enumerate}
\item $v(ab) = v(a) + v(b)$,
\item $v(a + b) \ge \min(v(a), v(b))$, with equality if $v(a)\neq v(b)$.
\end{enumerate}
\end{defn}
Let $(X,x)=({\rm Spec}(R),\mathfrak{m})$, we denote the space of valuations 
$${\rm Val}_{X,x}=\{\mbox{real valuations $v$ of $K$ with $v(f)>0$ for any $f\in \mathfrak{m}$}\}.$$
It has a natural topology (see \cite[Section 4.1]{JM12}).

If $(X,\Delta)$ is klt, following \cite[Section 5]{JM12}, we can define the function of log discrepancy $A_{X,\Delta}(v)$ on $ {\rm Val}_{X,x}$ extending the log discrepancy of divisorial valuations defined in Section \ref{s-dualcomplex}, and we denote by $ {\rm Val}^{=1}_{X,x}\subset {\rm Val}_{X,x}$  the subset consisting of all valuations with log discrepancy equal to 1. Similar to the global definition of volumes, we can also define a local volume of a valuation for $v\in \Val_{X,x}$ (see \cite{ELS})
$$\vol(v)=\lim \frac{{\rm length}(R/\fa_k)}{k^n/n!},$$
where $\fa_k=\{f\in R | \ v(f)\ge k\}$.
 
 \begin{defn}[\cite{Li15}]
 For any valuation $v\in  {\rm Val}_{X,x}$, we define the {\it normalised volume} $\hvol_{X,\Delta}(v)=(A_{X,\Delta}(v))^n\cdot \vol(v)$, and the volume of the klt singularity $x\in (X,\Delta)$ to be $\vol(x,X,\Delta)=\inf_{v\in \Val_{X,x}}\hvol(v)$.  By abuse of notation, we will often denote  $\vol(x,X,\Delta)$ by $\vol(x,X)$ if the context is clear.
 \end{defn}
 It is easy to see that $\hvol(v)=\hvol(\lambda v)$ for any $\lambda>0$, so that we can only consider the function $\hvol$ on $\Val^{=1}_{X,x}$.
 In \cite{Li15}, it was shown that $\vol(x,X)>0$. In \cite{Liu16}, a different characterisation 
 is given:
 \begin{equation}\label{Liu}
 \vol(x,X)=\inf_{\mathfrak{m}-{\rm primary\ } \fa } \mult(\fa)\cdot \lct(X,\Delta;\fa)^n . 
 \end{equation}
See \cite[9.3.14]{Laz04} for the definition of the log canonical threshold of a klt pair $(X,\Delta)$ with respect to an ideal $\fa$. 
Then in \cite{Blu16}, using an argument combining estimates on asymptotic invariants and the generic limiting construction, it is show that there always exists a valuation $v$ such that
$\vol(x,X)=\hvol(v)$, i.e., the infimum is indeed a minimum, confirming a conjecture in \cite{Li15}. Therefore the main questions left are two-fold.
\begin{ques}For a klt singularity $x\in (X,\Delta)$,
\begin{enumerate}
\item[I.] Characterise the geometric properties of the minimiser $v$.
\item[II.] Compute the volume $\vol(x,X)$. 
\end{enumerate}
\end{ques}
In what follows below, we will discuss these two questions in different sections.   

 \subsection{Geometry of the minimiser}
 In the recent birational geometry study of Fano varieties, it has become clear that the interplay between the ideas from higher dimensional geometry and the ideas from the complex geometry, centred around the study of K\"ahler-Einstein metrics, will lead to deep results. The common ground is the notion of K-(semi, poly)stability and  their cousin definitions (see e.g. \cite{Oda13,LX14,Fuj15} etc.).  An example is the construction of a proper moduli scheme parametrising the smoothable K-polystable Fano varieties (see e.g. \cite{LWX}). Although to establish a moduli space of Fano varieties is certainly a natural question to algebraic geometers, without a condition like K-stability with a differential geometry origin, such a functor does not behave well (e.g. the functor of smooth family Fano manifolds is not seperated.). Moreover the arguments used in the current construction of moduli spaces of K-polystable Fano varieties heavily depend on the results proved using analytic tools as in \cite{CDS, Tia15}. 
 
 Our main motivation to consider $v$ is to establish a `local K-stability' theory for klt singularities, guided by the local-to-global philosophy mentioned in the introduction. In particular, we propose the following conjecture for all klt singularities. 
 \begin{conj}[Stable Degeneration Conjecture, \cite{Li15, LX17b}]\label{conj-local}
 Given any arbitrary klt singularity $x\in (X={\rm Spec}(R),\Delta)$. There is a  unique minimiser $v$ up to rescaling. Furthermore, $v$ is quasi-monomial, with a finitely generated associated graded ring $R_0=_{\rm defn}{\rm gr}_v(R)$, and the induced degeneration
 $$(X_0={\rm Spec}(R_0),\Delta_0, \xi_v)$$ is a K-semistable Fano cone singularity. (See below for the definitions.)  \end{conj}
 
For the definition of quasi-monomial valuations, see \cite[Section 3]{JM12}. It is shown that they are the same as Abhyankar valuations (\cite[2.8]{ELS}).
From an arbitrary  quasi-monomial valuation $v \in \Val_{X,x}$, there is a standard process to degenerate ${\rm Spec}(R)$ to the associated graded ring ${\rm Spec}(R_0)$ over a complicated (e.g. non-Noetherian) base (see \cite{Tei03}). However, when $R_0$ is finitely generated, the degeneration can be understood in a much simpler way: we can embed ${\rm Spec}(R)$ into an affine space $\bb C^N$ of sufficiently large dimension, such that there exists a $\bb C^*$-action on $\bb C^N$ with a suitable weight $(\lambda_1,..., \lambda_N)$ satisfying that ${\rm Spec}(R_0)$ is the degeneration of ${\rm Spec}(R)$ under this one-parameter $\bb C^*$-action (see e.g. \cite{LX17b}). 
\bigskip
 
 The following example which predates our study is a prototype from the context of constructing Sasaki-Einstein metrics in Sasaki geometry.
\begin{exmp}[Fano cone singularity]\label{c-fanocone}
Assume that $X={\rm Spec}_{\bb C}(R)$ is a normal affine variety. Denote by $T$ a complex torus $(\bb C^*)^r$ which acts on $X$ faithfully. 
Let $N={\rm Hom}(\mathbb{C}^*, T)\cong \bb Z^{\oplus r}$ be the co-weight lattice and $M=N^*$ the weight lattice. We have a weight space decomposition 
\[
R=\bigoplus_{\alpha\in \Gamma} R_\alpha \mbox{  where \ } \Gamma =\{ \alpha\in  M  |\ R_{\alpha}\neq 0\}.
\]
We assume $R_{(0)}=\mathbb{C}$ which means there is a  unique fixed point $o$ contained in the closure of each orbit. 
Denote by $\sigma^{\vee}\subset M_{\mathbb{R}}$ the  convex cone generated by $\Gamma$, which is called the {\it weight cone} (or the {\it moment cone}).  
We define the {\it Reeb cone}
$$\mathfrak{t}^+_{\bb R}:=\{\ \xi \in N_{\mathbb{R}}\ \ | \  \langle \alpha, \xi \rangle>0 \mbox{ for any }\alpha\in \Gamma \}.$$
Then for any vector $\xi\in \mathfrak{t}^+_\bb R$ on $X$ we can associate a natural valuation $v_{\xi}$, which is given by 
$$v_{\xi}(f)= \min\{\langle \alpha, \xi \rangle \ |\  f_{\alpha}\neq 0 \mbox{ if we write }f=\sum f_{\alpha}\}.$$
If  $X$ have klt singularities, we call $(X,\xi)$ a {\it Fano cone singularity} for the following reason: for any $\xi\in N_{\bb Q}\cap \mathfrak{t}^+_{\bb R}$, then it generates a $\bb C^*$-action on $X$, and the quotient will be a log Fano pair as we assume $X$ is klt. 

 For isolated Fano cone singularities, minimising the normalised volume $\hvol$ among all valuations of the form $v_{\xi}$ ($\xi \in \mathfrak{t}^+_\bb R$) was initiated in the work \cite{MSY}, where $\hvol$ is defined analytically. It is shown there that the existence of a Sasaki-Einstein metric along $\xi_0$ implies $v_{\xi_0}$ is a minimiser among all $\xi \in \mathfrak{t}^+_\bb R$. Moreover, it is proved that $\hvol$ is strictly convex on $ \mathfrak{t}^+_\bb R$, which is  an evidence for the claim of  the uniqueness in Conjecture \ref{conj-local}. 
 
 Later, in \cite{CS15}, following \cite{Tia97, Don01}, the K-(semi)polystability was formulated for Fano cone singularities. It was a straightforward calculation from the definition to show that if $(X,\xi_0)$ is K-semistable, then $v_{\xi_0}$ is a minimiser among all valuations of the form $v_{\xi}$ for $\xi\in  \mathfrak{t}^+_\bb R$. However, it takes  significant more work in \cite{LX17b} to show that if $(X,\xi_0)$ is K-semistable, then $v_{\xi_0}$ is a minimiser in the much larger space $ \Val_{X,x}$ and  unique among all quasi-monomial valuations up to rescaling (see Step 3 and 6 in the sketch of the proofs of Theorem \ref{t-kollar} and \ref{t-high} below). 
\end{exmp}
 
 \bigskip
 
 For a  normal singularity $x\in (X={\rm Spec}(R),\Delta)$ with a quasi-monomial valuation $v\in \Val_{X,x}$ of rational rank $r$, we assume that its associated graded ring $R_0$ is finitely generated. By the grading,  ${\rm Spec}(R_0)$ admits a torus $T\cong (\mathbb{C}^*)^r$-action, thus we can put it in (a log generalisation of) the setting of Example \ref{c-fanocone} as follows.    Let $\Phi$ be the valuative semigroup of $v$, then it generates a group $\Phi^{\rm g}\cong \bb Z^r$ which is isomorphic to the weight lattice $M=N^*$. Under this isomorphism the weight cone is generated by $\alpha\in \Phi$. Since the embedding $\iota_v\colon \Phi^{\rm g}\to \bb R$ restricts to $\iota_v^+\colon \Phi\to \mathbb{R}_{+}$, it yields a vector in the Reeb cone $ \mathfrak{t}^+_\bb R\subset N_{\bb R}$, denoted by $\xi_v$. 
 Let $\Delta_0$ be the natural divisorial degeneration of $\Delta$ on $X_0={\rm Spec}(R_0)$.
We call such a valuation $v\in \Val_{x,X}$  to be {\it K-semistable}, if  $(X_0,\Delta_0,\xi_v)$ is a K-semistable Fano cone. In particular, we require $(X_0,\Delta_0)$ to be klt. 
 Since a K-semistable valuation is always a minimiser (see Theorem \ref{t-high}), Conjecture \ref{conj-local} predicts that for any klt singularity $x\in (X,\Delta)$, the minimiser of $\hvol$ is precisely the same as the notion of a K-semistable valuation. 

\bigskip
 
 We have established various parts of Conjecture \ref{conj-local}. First we consider the case that the minimiser is a divisorial valuation.
 \begin{thm}[{\cite[Theorem 1.2]{LX16}}]\label{t-kollar}
 Let $x\in (X,\Delta)$ be a klt singularity. If a divisorial valuation $\ord_S\in \Val_{X.x}$ minimises the function $\hvol_{X,\Delta}$, then $S$ is a Koll\'ar component over $x$, and the induced log Fano pair $(S,\Delta_S)$ is K-semistable. Furthermore, $\hvol(\ord_S)<\hvol(\ord_E)$ for any divisor $E\neq S$ centred on $x$.
 
 Conversely, if $S$ is a Koll\'ar component centred on $x$ such that the induced log Fano pair $(S,\Delta_S)$ is K-semistable, then $\ord_S$ minimises $\hvol_{X,\Delta}$.
 \end{thm}
 An immediate consequence is that, if instead of searching general Koll\'ar components, we only look for the semi-stable ones, then if one exists, it is unique. In general, Conjecture \ref{conj-local} predicts that if we choose a sequence of rational vectors $v_i \in  \mathfrak{t}^+_{\bb Q}$ that converge to $v$, then the quotient   of $X_0$ by the $\mathbb{C}^*$-action along $v_i$ induces a Koll\'ar component $S_{i}$ centred on $x\in (X,\Delta)$ which satisfies $c_i\cdot \ord_{S_i}\to \xi_v$ after a suitable rescaling (cf. \cite{Li15b, LX17b}). In \cite{LX16}, we confirm that any minimiser is always a limit of a sequence of Koll\'ar components with a suitable rescaling. 
 
 In general, a quasi-monomial valuation with higher rational rank could appear as the minimiser (cf. \cite{Blu16}). In this case, we can also prove the following result. 
 
  \begin{thm}[{\cite[Theorem 1.1]{LX17b}}]\label{t-high}
Let $x\in (X,\Delta)$ be a klt singularity.  Let $v$ be a quasi-monomial valuation in $\Val_{X,x}$ that minimises $\hvol_{(X,\Delta)}$ and has a  finitely generated associated graded ring  ${\rm gr}_v(R)$. Then the following properties hold:
\begin{enumerate}
\item[(a)] The degeneration
$\big(X_0=_{\rm defn}{\rm Spec}\big({\rm gr}_v(R)\big), \Delta_0, \xi_v \big)$ is a K-semistable Fano cone, i.e. $v$ is a K-semistable valuation;
\item[(b)] Let $v'$ be another quasi-monomial valuation in $\Val_{X,x}$ that minimises $\hvol_{(X,\Delta)}$. Then $v'$ is a rescaling of $v$.
\end{enumerate}
Conversely, any quasi-monomial valuation that satisfies (a) above is a minimiser. 
 \end{thm}
 
\begin{proof}[Sketch of ideas in the Proofs of Theorem \ref{t-kollar} and \ref{t-high}] The proof consists of a few steps, involving different techniques. 

\medskip

\noindent {\it Step 1}: In this step, we illustrate how Koll\'ar components come into the picture.  
From each ideal $\fa$, we can take a dlt modification of 
$$f\colon (Y,\Delta_Y)\to (X,\Delta+\lct(X,\Delta;\fa)\cdot\fa),$$ where $\Delta_Y=f_*^{-1}\Delta+{\rm Ex}(f)$ and for any component $E_i\subset {\rm Ex}(f)$ we have $$A_{X,\Delta}(E)=\lct(X,\Delta;\fa)\cdot \mult_{E}f^*\fa.$$ 
There is a natural inclusion $\D(\Delta_Y)\subset \Val^{=1}_{X,x}$, and using a similar argument as in \cite{LX14}, we can show that there exists a Koll\'ar component $S$ whose rescaling in $\Val^{=1}_{X,x}$ contained in $\D(\Delta_Y)$ satisfies that 
$$\hvol(\ord_S)= \vol^{\rm loc}(-A_{X,\Delta}(S)\cdot S)\le \vol^{\rm loc}(-K_Y-\Delta_Y)\le \mult(\fa)\cdot \lct^n(X,\Delta;\fa).$$
Then \eqref{Liu} implies that
$$\vol(x,X)=\inf\{\hvol(\ord_S)|\  S \mbox{ is a Koll\'ar component}\}.$$
We can also show that if a minimiser is a divisor then it is indeed a Koll\'ar component (this is proved independently in \cite{Blu16}). 

Moreover, if $x\in (X,\Delta)$ admits a torus group $T$-action, then by degenerating to the initial ideals, as the colengths are preserved and the log canonical thresholds  may only decrease, the right hand side of \eqref{Liu} can be replaced by all $T$-equivariant ideals. Moreover, equivariant MMP allows us to make all the above data $Y$ and $S$ $T$-equivariant.

\medskip
\noindent {\it Step 2}:  In this step, we show that if a minimiser $v$ is quasi-monomial such that $R_0={\rm gr}_v(R)$ is finitely generated, then the degeneration pair $(X_0=_{\rm defn} {\rm Spec}(R_0),D_0)$ is klt. After Step 1, this is easy in the case of Theorem \ref{t-kollar}, as the Koll\'ar component is klt. To treat the higher rank case in Theorem \ref{t-high}, we verify two ingredients: first we show that a rescaling of $\ord_{S_i}$ for the approximating sequence of $S_i$ in Step 1 can be all chosen in the dual complex of a fixed model considered as a subspace of  $ \Val^{=1}_{X,x}$; then we show as ${\rm gr}_v(R)$ is finitely generated, for any $i$ sufficiently large, ${\rm gr}_v(R)\cong {\rm gr}_{\ord_{S_i}}(R)$. This immediately implies that $(X_0,D_0)$ is the same as the corresponding cone $C(S_0,\Delta_{S_0})$ over the Koll\'ar component $S_0$, and then we conclude it is klt as before. 

\medskip
\noindent {\it Step 3}: To proceed we need to establish properties of a general log Fano cone $(X_0,\Delta_0,\xi_v)$ and show that the  corresponding valuation $v$ is a minimiser if and only if $(X_0,\Delta_0,\xi_v)$ is K-semistable. First assume $(X_0,\Delta_0,\xi_v)$ is K-semistable, then by Step 1, it suffices to show that for any $T$-equivariant Koll\'ar component $S$, $\hvol(\ord_{S})\ge \hvol(v)$. In fact, for any such $S$, it induces a special degeneration of $(X_0,\Delta_0,\xi_v)$ to $(Y,\Delta_Y, \xi_Y)$ admitting a $((\bb C^*)^{r}\times \bb C^*)$-action and a new rational vector $\eta_S\in N\oplus \bb Z$ corresponding to the $\bb C^*$-action on the special fiber induced by the degeneration.  Then an observation going back to \cite{MSY} says that 
$$\frac{d\ \hvol(\xi_Y+t\cdot \xi_S)}{dt}={\rm Fut}(Y,\Delta_Y, \xi_Y;\xi_S)\ge 0.$$
Here the generalised Futaki invariant ${\rm Fut}(Y,\Delta_Y, \xi_Y;\xi_S)$ is defined in   \cite[2.2]{CS15}, and then  the last inequality comes from the K-semistability assumption.   It is also first observed in \cite{MSY} that the normalised volume function $\hvol$ is convex on the space of valuations 
$\{\xi_v \ | \ v\in \mathfrak{t}^+_{\bb R} \}.$
Thus by restricting the function  on the ray $\xi_Y+t\cdot\xi_S$ $(t\ge 0)$ and applying the convexity, we conclude that 
$$\hvol_{X_0}(\ord_S)=\lim_{t\to \infty} \hvol_{Y}(\xi_Y+t \cdot \xi_S)\ge \hvol_Y(\xi_Y)=\hvol_{X_0}(\xi_v).$$
 Reversing the argument, one can show that if $v$ is a minimiser of $\hvol$ for a log Fano cone singularity $(X_0,\Delta_0,\xi_v)$, then for any special degeneration with the same notation as above, we have ${\rm Fut}(Y,\Delta_Y, \xi_Y;\xi_S)\ge 0$. 

\medskip

\noindent {\it Step 4:} An consequence of  Step 3 is that for a valuation $v$ on $X$ such that the degeneration $(X_0,\Delta_0,\xi_v)$ is K-semistable,  since the degeneration to the initial ideal argument implies that $\vol(x,X)\ge \vol(o, X_0)$,  then 
$$\hvol_X(v)=\hvol_{X_0}(\xi_v)=\vol(o,X_0)$$ is equal to $\vol(x,X)$. 

\medskip
\noindent {\it Step 5}: Then we proceed to show that if a log Fano cone $(X_0,\Delta_0,\xi_v)$ comes from a degeneration of a minimiser as in Step 2, then it is K-semistable. If not, by Step 3, we can find a degeneration $(Y,\Delta_Y, \xi_{Y})$ induced by an equivariant Koll\'ar component $S$ with $\hvol_Y(\ord_S)<\hvol_Y(\xi_Y)=\hvol_{X_0}(\xi_v)$. Then arguments similar to \cite[Section 5]{And13} show we can construct a degeneration of $(X,\Delta)$ to $(Y,\Delta_Y)$ and a family of valuations $v_t\in \Val_{X,x}$ for $t\in [0,\epsilon]$ (for some $0<\epsilon\ll 1$), with the property that
$$\hvol_X(v_t)=\hvol_Y(\xi_Y+t\cdot\xi_S)<\hvol_Y(\xi_Y)=\hvol_{X_0}(\xi_v)=\hvol_X(v),$$
where for the second inequality, we use again the fact that $\hvol_Y(\xi_Y+t\cdot\xi_S)$ is a convex function. But this is a contradiction.

\medskip
\noindent {\it Step 6}: Now we turn to the uniqueness. In this step, we show this for a K-semistable Fano cone singularity $(X_0,\Delta_0,\xi_v)$. In fact, for any $T$-equivariant valuation $\mu$, we can connect $\xi_v$ and $\mu$ by a path $\mu_t$ such that $\mu_0=\xi_v$ and $\mu_1=\mu$. A Newton-Okounkov body type construction (similar to \cite{KK14}) can interpret the volumes $\hvol(\mu_t)$ to be the volumes of the regions $\Cal A_t$ contained in the convex cone $\Cal C$ cut out by a hyperplane $H_t$ passing through a given vector inside $\Cal C$. Then we conclude by the fact in the convex geometry which says that such a function $f(t)=\vol(\Cal A_t)$ is strictly convex.  Thus it has a unique minimiser, which is $\xi_v$ by Step 3. 

\medskip
\noindent {\it Step 7}: The last step is to prove the uniqueness in general, under the assumption that it admits a degeneration $(X_0,\Delta_0,\xi_v)$ given by a K-semistable minimiser $v$. For another quasi-monomial minimiser $v'$ of rank $r'$, by a combination of the Diophantine approximation and an MMP construction including the application of ACC of log canonical thresholds (see Section \ref{s-acc}),  we can obtain a model $f\colon Z\to X$ which extracts $r'$ divisors $E_i$ ($i=1,...,r'$) such that $(Z, \Delta_Z=_{\rm defn}\sum E_i+f_*^{-1}\Delta)$ is log canonical. Moreover,  the quasi-monomial valuation $v'$ can be computed at the generic point of a component of the intersection of $E_i$, along which $(Z,\Delta_Z)$ is toroidal. 
Then with the help of the MMP, a careful analysis can show  $Z\to X$ degenerates to a birational morphism $Z_0\to X_0$. Moreover, there exists a quasi-monomial valuation  $w$ computed on $Y_0$ which can be considered as a degeneration of $v'$ with 
$$\hvol_{X_0}(w)=\hvol_X(v')=\hvol_X(v)=\hvol_{X_0}(\xi_v).$$ Thus $w=\xi_v$ by Step 5 after a rescaling. Since $w({\bf in}(f))\ge v'(f)$ and $\vol(w)=\vol(v')$, we may argue this implies $\xi_v({\bf in}(f))=v'(f)$. Therefore, $v'$ is uniquely determined by $\xi_v$. 
\end{proof}

Weaker than Theorem \ref{t-kollar},  in Theorem \ref{t-high} we can not show the finite generation of  ${\rm gr}_v(R)$, thus we have to post it as an assumption. This is due to the fact that unlike in the divisorial case where the construction of Koll\'ar component provides a satisfying birational model to understand $\ord_S$, for a  quasi-monomial valuation of higher rank, the auxiliary models (see Step 1 and 6 in the above proof) we construct are less canonical.   Moreover, compared to the statement in Conjecture \ref{conj-local}, it remains wide open to verify that the minimiser is always quasi-monomial. 

\medskip
 
 One of the main applications of Theorem \ref{t-kollar} and \ref{t-high} is to address Donaldson-Sun's conjecture in \cite{DS15} on the algebraicity of the construction of the metric tangent cone, which can be considered as a local analogue of \cite{DS14, Tia13, LWX}.  More precisely, it was proved that the Gromov-Hausdorff limit of a sequence of K\"ahler-Einstein metric Fano varieties is a Fano variety $X_{\infty}$ with klt singularities (see \cite{DS14, Tia13}). And to understand the metric structure near a singularity $x\in X_{\infty}$, we need to understand its  metric tangent cone $C$ (cf. \cite{CCT02}).   In the work \cite{DS15},  a description of $C$ was given by  a two-step degeneration process: first there is a valuation $v$ on $\Val_{X_{\infty},x}$ whose associated graded ring induces a degeneration of $x\in X_{\infty}$ to $o\in M$; then there is a degeneration of Fano cone from $o\in M$ to $o'\in C$. In Donaldson-Sun's definitions of $M$ and $C$, they used the local metric structure around $x\in X_{\infty}$. However, they conjectured that both $M$ and $C$ only depend on the underlying algebraic structure of the germ $x\in X_{\infty}$.
 Built on the previous works of \cite{Li15b, LL16, LX16}, we answer the first part of their conjecture affirmatively, which says $M$ is determined by the algebraic structure of the germ $x\in X_{\infty}$. We achieve this by showing that $v$ is a K-semistable valuation in $\Val_{X,x}$ and such a K-semistable valuation is unique up to rescaling.
 \begin{thm}[\cite{LX17b}]
The valuation $v$ is the unique minimiser (up to scaling) of $\hvol$ in all quasi-monomial valuations in $\Val_{X_{\infty},x}$. 
 \end{thm}
 \begin{proof}From the results proved in \cite{DS15}, we can verify that $o\in (W,\xi_v)$ is a K-semistable Fano cone singularity, which exactly means $v$ is a K-semistable valuation. Thus $v$ is a minimiser of $\hvol$ by the last statement of Theorem \ref{t-high}. Then up to rescaling, $v$ is the unique quasi-monomial minimiser again by Theorem \ref{t-high}. 
 \end{proof}
 
 We expect that the tools we developed, especially those on equivariant K-stability, are enough to solve the second part of Donaldson-Sun's conjecture, i.e. to confirm the metric tangent cone $C$ only depends on the algebraic structure of $x\in X_{\infty}$.

 \subsection{The volume of a klt singularity}
 
As $\vol(x,X)$ carries deep information on the singularity $x\in X$,  calculating this number consists an important part of the theory. It also has applications to global questions. We discuss some related results and questions in this section. 

In general, it could be difficult to compute $\vol(x,X)$. Even for the smooth point $x\in \mathbb{C}^n$, knowing $\vol(x,\mathbb{C}^n)$ (which is, not surprisingly, equal to $n^n$) involves highly nontrivial arguments. An illuminating example is the following.
\begin{exmp}[\cite{Li15b,LL16, LX16}] A $\bb Q$-Fano variety is K-semistable if and only if for the cone $C=C(X,-rK_X)$, the canonical valuation $v$ obtained by blowing up the vertex $o$ is a minimiser. 
\end{exmp}
On one hand, this means that finding out the minimiser is in general  at least as hard as testing the K-semistablity of (one dimensional lower) Fano varieties, which has been known to be a challenging question; on the other hand, this sheds new light on the question of testing K-stability. For example, using properties of degenerating ideals to their initials, we can prove that for a klt Fano variety $X$ with a torus group $T$-action, to test the K-semistability of $X$ it suffices to test on $T$-equivariant special test configurations (see \cite{LX16}). 

\medskip

The Stable Degeneration Conjecture \ref{conj-local} implies many properties of $\vol(x,X)$. The first one we want to discuss  is a finite degree multiplication formula. 
\begin{conj}\label{conj-finite}
If $\pi\colon x_1\in (X_1,\Delta_1)\to x_2\in (X_2,\Delta_2)$ is a finite dominant morphism between klt singularities such that $\pi^*(K_{X_2}+\Delta_2)=K_{X_1}+\Delta_1$, then 
$$\deg(\pi)\cdot \vol(x_2, X_2)=\vol(x_1, X_1).$$
\end{conj}
This can be easily reduced to the case that the finite covering $X_1\to X_2$ is Galois, and we denote the Galois group by $G$. Then it suffices to show that the minimiser of $X_1$ is $G$-equivariant, which is implied by the uniqueness claim in Conjecture \ref{conj-local}. Conjecture \ref{conj-finite} is verified in \cite{LX17b} for $x\in X_{\infty}$ where $X_{\infty}$ is a Gromov-Hausdorff limit of K\"ahler-Einstein Fano manifolds. Since any point will have its volume less or equal to $n^n$ (see \cite[Appendix]{LX17a}), Conjecture \ref{conj-finite} implies that for a klt singularity $x\in (X,\Delta)$, 
\begin{eqnarray}\label{e-finite}
\vol(x,X)\le n^n/|\hat{\pi}^{\rm loc}_1(x,X)|,
\end{eqnarray}
where the finiteness of $\hat{\pi}^{\rm loc}_1(x,X)$ is proved in \cite{Xu14}. 

\medskip

Combining Conjecture \ref{conj-local} with the well known speculation that K-semistable is a Zariski open condition, we also have the following conjecture.

\begin{conj}\label{conj-LSC}
Given a klt pair $(X,\Delta)$, then the function $\vol(x,X)$ is a constructible function, i.e. we can stratify $X$ into constructible sets $X=\sqcup_i S_i$, such that for any $i$, $\vol(x, X)$ takes a constant value for all $x\in S_i$. 
\end{conj} 

A degeneration argument in \cite{Liu17} implies that the volume function should be lower semi-continuous. A special case we know is that the volume of any $n$-dimensional klt non-smooth point is always less than $n^n$ (see \cite[Appendix]{LX17a}).

\medskip

Finally, we discuss some applications of the volume of singularities  to K-stability of Fano varieties. A useful formula connecting local and global geometries is the following.
\begin{thm}[\cite{Fuj15,Liu16}]If $X$ is a K-semistable $\bb Q$-Fano variety, then for any point $x\in X$,
we have 
\begin{eqnarray}\label{e-upper}
\vol(x, X)\ge (\frac{n+1}{n})^n(-K_X)^n.
\end{eqnarray}
\end{thm}
So if  we can bound the type of klt singularities from the lower bound of their volumes, then we can restrict the type of singularities appearing on a K-semistable $\mathbb{Q}$-Fano variety with a given volume. In particular,   this applies to the Gromov-Hausdorff limit $X_{\infty}$ of a sequence of K\"ahler-Einstein Fano manifolds $X_i$ (with a large volume of $-K_{X_i}$). If the restriction is sufficiently effective, then $X_{\infty}$ would appear in an explicit simple ambient space on which we can carry out the orbital geometry calculation to identify $X_{\infty}$ by showing all other possible limits are K-unstable.

 For instance,  by revisiting the classification  results  of three dimensional singularities, we show that $\vol(x,X)\le 16$ if $x\in X$ is singular and the equality holds if and only if $x\in X$ is a rational double point (see \cite{LX17a}).  As a consequence, we could solve the question on the existence of K\"ahler-Einstein metrics for cubic threefolds. 
\begin{cor}[\cite{LX17a}]GIT polystable (resp. semistable) cubic threefolds are K-polystable (resp. K-semistable). In particular, all GIT polystable cubic threefolds, including every smooth one, admit K\"ahler-Einstein metrics. 
\end{cor} 


\bigskip

\end{document}